\documentclass[11pt]{article}
\usepackage{amsfonts,comment}
\usepackage{amssymb}
\usepackage{graphicx}
\usepackage{amsmath}
\usepackage{amsthm}
\usepackage{tikz}

\usepackage[affil-it]{authblk}
\usepackage{xcolor}

\nonstopmode
\newcommand{\BlackBoxes}{\global\overfullrule5pt}

\BlackBoxes

\newtheorem{theorem}{Theorem}[section]
\newtheorem{definition}[theorem]{Definition}
\newtheorem{proposition}[theorem]{Proposition}

\newtheorem{lemma}[theorem]{Lemma}

\newtheorem{remark}[theorem]{Remark}
\newtheorem{example}[theorem]{Example}
\newtheorem{examples}[theorem]{Examples}

\newtheorem{foo}[theorem]{Remarks}
\newtheorem{Pre}[theorem]{}

\newcommand{\R}{{\mathbb R}}
\newcommand{\N}{{\mathbb N}}

\newcommand{\Acal}{{\mathcal A}}

\newcommand{\Ccal}{{\mathcal C}}
\newcommand{\Dcal}{{\mathcal D}}

\newcommand{\Fcal}{{\mathcal F}}
\newcommand{\Gcal}{{\mathcal G}}
\newcommand{\Hcal}{{\mathcal H}}

\newcommand{\Lcal}{{\mathcal L}}
\newcommand{\Mcal}{{\mathcal M}}

\newcommand{\Pcal}{{\mathcal P}}
\newcommand{\Qcal}{{\mathcal Q}}

\newcommand{\Scal}{{\mathcal S}}

\newcommand{\Zcal}{{\mathcal Z}}

\renewcommand\inf{\qopname\relax m{\vphantom{p}inf}}

\begin{document}

\title{Fatou Closedness under Model Uncertainty 
}
\author{Marco Maggis\thanks{Email: \texttt{marco.maggis@unimi.it}}\,,
}
\affil{Department of Mathematics, University of Milan, Italy}
\author{
Thilo Meyer-Brandis\thanks{Email: \texttt{meyerbra@math.lmu.de}},  Gregor Svindland\,\thanks{Email: \texttt{svindla@math.lmu.de}}
}
\affil{Mathematics Institute, LMU Munich, Germany}

\date{February 28, 2018}

\maketitle

\begin{abstract}
We provide a characterization in terms of Fatou closedness for
weakly closed monotone convex sets in the space of $\Pcal$-quasisure
bounded random variables, where $\Pcal$ is a (possibly
non-dominated) class of probability measures. Applications of our results lie within robust versions the Fundamental Theorem of Asset Pricing or dual representation of convex risk measures.
\end{abstract}



\section{Introduction}

A fundamental result attributed to Grothendieck (\cite[p321, Exercise~1]{Groth}) and based on the Krein-Smulian theorem 
characterizes weak*-closedness of a convex subset of 
$L_P^\infty:=L^\infty(\Omega, \Fcal,P)$, where
$(\Omega,\Fcal,P)$ is a probability space, by means of a property
called Fatou closedness as follows:

\begin{theorem}\label{thm:basic}
Let $\Acal\subset L_P^\infty$ be convex. Equivalent are:
\begin{itemize}
 \item[(i)] $\Acal$ is weak*-closed (i.e.\ closed in $\sigma(L_P^\infty, L_P^1)$).
 \item[(ii)] $\Acal$ is Fatou closed, i.e.\ if $(X_n)_{n\in \N}\subset \Acal$
 is a bounded sequence which converges $P$-almost surely to $X$, then $X\in \Acal$.
\end{itemize}
\end{theorem}

Note that $L_P^\infty$ is a Banach lattice (see Section~\ref{notations}) and that from this point of view property~(ii) in Theorem~\ref{thm:basic} equals sequential order closedness of $\Acal$ which in fact implies order closedness since $L_P^\infty$ has the countable sup property, i.e.\ every nonempty subset possessing a supremum contains a countable
subset possessing the same supremum. Theorem~\ref{thm:basic}  is very useful and often applied in the mathematical finance literature such as in the classic proof of the Fundamental Theorem
of Asset Pricing, see e.g.\ \cite{DS94} or \cite{DS06}, or in the
dual representation of convex risk functions, see e.g.\ \cite{FS3}. In
all cases the problem is that the norm dual of $L_P^\infty$
contains undesired singular elements, whereas in the weak*-duality
$(L_P^\infty,\sigma(L_P^\infty,L_P^1))$ the elements of the dual
space are identified with $\sigma$-additive measures. However, as
the weak*-topology is generally not first-countable, verifying that some set
is weak*-closed is typically quite challenging. This is where
Theorem~\ref{thm:basic} proves helpful.

The aim of this paper is to study the existence of a version of
Theorem~\ref{thm:basic} for the case when the probability measure $P$ is
replaced by a class $\Pcal$ of probability measures on
$(\Omega,\Fcal)$. In general this class $\Pcal$ does not allow for
a dominating probability. Applications of such a result lie for
instance in the field of mathematical finance, where currently there is much attention paid to deriving versions of the Fundamental Theorem of Asset
Pricing  as well as dual
representations of convex risk functions in so-called {\em robust} frameworks as studied
in  \cite{B13, Bion,BN13,BFM15,Nu14, V14}.
These kind of frameworks have become increasingly popular to describe a decision maker who has
to deal with the uncertainty which arises from model ambiguity. Here the class of probability models $\Pcal$ the decision maker takes into account represents her degree of ambiguity about the right probabilistic model. If $\Pcal=\{P\}$ there is no ambiguity. In many studies which account for model ambiguity $\Pcal$ in fact turns out to be a non-dominated class of probability measures, see  \cite{Bion,BN13,BFM15,Nu14} and the reference therein.

We will show that  there is a version of
Theorem~\ref{thm:basic} in a robust probabilistic
framework $(\Omega,\Fcal,\Pcal)$, see Theorem~\ref{thm:5}. Let $$c(A):=\sup_{P\in \Pcal} P(A), \quad A\in \Fcal,$$ denote the capacity generated by $\Pcal$. Under some conditions on the convex set $\Acal$ and on $L^\infty_c$ we obtain
equivalence between
\begin{description}
\item[(WC)] $\Acal\subset L^\infty_c$ is
$\sigma(L^\infty_c,ca_c)$-closed, \item[(FC)] $\Acal\subset
L^\infty_c$ is Fatou closed: for any bounded sequence
$\{X_n\}\subset \Acal$ and $X\in L^\infty_c$ such that $X_n\to X$
$\Pcal$-quasi surely we have that  $X\in \Acal$,
\end{description}
where $L^\infty_c$ and $ca_c$ are the robust analogues of
$L^\infty_P$ and $L^1_P$ given by the capacity $c$, respectively, and $\Pcal$ quasi sure
convergence means $Q$-almost sure convergence under each $Q\in
\Pcal$. The conditions we have to require on $\Acal$ are monotonicity ($\Acal=\Acal+(L^\infty_c)_+$) and
a property called $\Pcal$-sensitivity. Monotonicity is
typically satisfied in economic applications, and we
show that $\Pcal$-sensitivity is indeed a necessary condition to
have (WC) $\Leftrightarrow$ (FC), see Proposition~\ref{prop:sensitive}. If $\Pcal$ is dominated,
$\Pcal$-sensitivity is always fulfilled.

Another requirement which is crucial for our proof of (WC) $\Leftrightarrow$ (FC) is that the dual space of $ca_c$ may be identified with $L^\infty_c$. 
This condition implies order completeness of the Banach lattice $L^\infty_c$, i.e.\ the existence of a supremum for any bounded subset of $L^\infty_c$, see Proposition~\ref{prop:esssup}, and it thus corresponds to aggregation type results as in
\cite{Cohen,STZ11}. If $L^\infty_c$ is order complete, then the property (FC) equals sequential order closedness of $\Acal$. However, order completeness does not imply that $L^\infty_c$ possesses the countable sup property, see Example \ref{ex:cs2}, so even under this condition (FC) does in general not imply order closedness of $\Acal$.

We also provide a counter example showing that for
non-dominated $\Pcal$ there is no proof of (WC) $\Leftrightarrow$
(FC) without further requirements such as $\Pcal$-sensitivity, see Example~\ref{ex:non}. Moreover,  we illustrate that many conditions, in particular on $\Pcal$,
one would think of in the first place to ensure (WC) $\Leftrightarrow$ (FC), indeed imply that $\Pcal$ is dominated, so we are back to Theorem~\ref{thm:basic}.
Hence, a further contribution
of this paper is to provide a deeper insight into the fallacies
one might encounter when attempting to extend
Theorem~\ref{thm:basic} to a robust case.

The paper is structured as follows: Section~\ref{notations}
provides a list of useful notations which will be used throughout
the paper. Section~\ref{robustFatou} contains the main
results of the paper, and in particular Theorem~\ref{thm:5} is the
robust version of Theorem~\ref{thm:basic}.  Finally, applications
of Theorem~\ref{thm:5} in the field of mathematical finance 
are collected in Section~\ref{applications}. Here we do not assume that the reader is familiar with
mathematical finance. However, we try to keep the presentation concise, referring to the relevant literature for more background information.

\section{Notation}\label{notations}

For the sake of clarity we propose here a list of the basic
notations and definitions that we shall use throughout this paper.

Let $(\Omega,\Fcal)$ be any measurable space.

\begin{enumerate}

\item $ba:=\{\mu:\Fcal\to \R\mid \mbox{$\mu$ is finitely
additive}\}$ and $ca:=\{\mu:\Fcal\to \R\mid \mbox{$\mu$ is
$\sigma$-additive}\}$. These are both Banach lattices once endowed
with the total variation norm $TV$ and $|\mu|=\mu^++\mu^-$
where $\mu=\mu^+-\mu^-$ is the Jordan decomposition (see
\cite{AB06} for further details).

\item $ba_+$ (resp. $ca_+$) is the set of all positive additive
(resp. $\sigma$-additive) set functions on $(\Omega,\Fcal)$.

\item In absence of any reference probability measure we have the
following sets of random variables
\begin{eqnarray*}
\Lcal & := & \{f:\Omega\to \R\mid f\, \mbox{is}\, \mbox{$\Fcal$-measurable}\},
\\ \Lcal_+ & := & \{f\in \Lcal\mid f(\omega)\geq 0\ , \forall\, \omega\in\Omega\},
\\ \Lcal^\infty & := & \{f\in \Lcal\mid \mbox{$f$ is bounded}\}.
\end{eqnarray*}
In particular $\Lcal^\infty$ is a Banach space under the
(pointwise) supremum norm $\|\cdot\|_{\infty}$ with dual space
$ba$.

\item $\Mcal_1\subset ca_+$ is the set of all probability measures
on $(\Omega,\Fcal)$.

\item Throughout this paper we fix a set of probability measures
$\Pcal\subset \Mcal_1$.

\item  We introduce the sublinear
expectation $$c(f):=\sup_{Q\in \Pcal}E_Q[f], \quad f\in \Lcal_+$$
and by some abuse of notation we define
the capacity $c(A):=c(1_A)$ for
$A\in \Fcal$.

\item  Let $\widehat \Pcal,\widetilde\Pcal\subset \Mcal_1$. $\widehat \Pcal$ dominates $\widetilde\Pcal$, denoted by $\widetilde\Pcal
\ll \widehat \Pcal $, if for all $A\in \Fcal$: $$\sup_{P\in \widehat\Pcal} P(A)=0 \quad \Rightarrow \quad  \sup_{P\in \widetilde \Pcal} P(A)=0 .$$
We say that two classes $\widehat\Pcal$ and $\widetilde\Pcal$
are equivalent, denoted by $\widehat \Pcal \thickapprox \widetilde\Pcal$,
if $\widetilde\Pcal\ll \widehat\Pcal $ and $\widehat\Pcal\ll \widetilde \Pcal$.

\item A statement holds $\Pcal$-quasi surely (q.s.) if the
statement holds $Q$-almost surely (a.s.) for any $Q\in \Pcal$.

\item The space of finitely additive (resp. countably additive)
set functions dominated by $c$ is given by $ba_c=\{\mu\in ba \mid
\mu \ll c \}$ (resp. $ca_c=\{\mu\in ca \mid \mu \ll c \}$). Here
$\mu\ll c$ means: $c(A)=0$ for
some $A\in \Fcal$ implies $\mu(A)=0$.
\\When $\Pcal=\{Q\}$ we shall write $ba_Q$ or $ca_Q$ for the sake of
simplicity.

\item We consider the quotient space $L_c:=\Lcal_{/\sim}$ where the
equivalence is given by
$$f\sim g\quad \Leftrightarrow \quad
\forall P\in \Pcal: P(f=g)=1 .
$$
We shall use capital letters to distinguish equivalence classes of
random variables $X\in L_c$ from a representative $f\in X$, with
$f\in \Lcal$. In case $\Pcal=\{Q\}$ we
shall write $L_Q^1$ instead of $L_c$. It is a well-known consequence of the Radon-Nikodym theorem (\cite[Theorem~13.18]{AB06}) that $ca_Q$ may be identified with $L_Q^1$.

\item For any $f,g\in\Lcal$ and $P\in\Mcal_1 $, we write $f\leq g$
$P$-a.s.\ if and only if $P(f\leq g)=1$. Similarly $f\leq g$ $\Pcal$-q.s.\ if
and only if $f\leq g$ $P$-a.s. for all $P\in \Pcal$. This relation is a partial
order on $\Lcal$ and  it also induces a partial order on $L_c$ where $X\leq Y$ for $X,Y\in L_c$ if and only if  $f\leq g$ $\Pcal$-q.s.\ for any $f\in X$ and $g\in Y$.

\item We define $L_c^\infty:=\Lcal^\infty_{/\sim}$ and endow this space with the norm
$$\|X\|_{c,\infty}:=\inf\{m\mid \forall P\in \Pcal:P(|X|\leq
m)=1\}.$$ $(L_c^\infty, \|\cdot\|_{c,\infty})$ is a Banach lattice
with the same partial order $\leq$ as on $L_c$. Its norm dual is $ba_c$.
 In case $\Pcal=\{Q\}$ we
shall write $L_Q^\infty$ and
$\|\cdot\|_{Q,\infty}$ for the sake of simplicity. Note that $\|\cdot\|_{c,\infty}$ is never order continuous for any choice of $\Pcal$.

\end{enumerate}

For simplicity of presentation, if there is no risk of confusion, we will follow the usual convention of identifying random variables in $\Lcal$ with the equivalence classes they induce (in $L_c$, $L^\infty_c$, $L^1_Q$ or $L_Q^\infty$) and vice versa.

\section{Towards a robust version of
Theorem~\ref{thm:basic}}\label{robustFatou}

We start by recalling the proof of the non-trivial implication (ii) $\Rightarrow$ (i) of Theorem~\ref{thm:basic}:
the idea is to apply the Krein-Smulian theorem which implies that
we only need to show that the sets $$C_K:=\Acal\cap \{X\in L^\infty_P\mid
\|X\|_{P,\infty} \leq K\}$$ are weak*-closed for any constant
$K>0$. Now we could invoke the countable sup property of $L^\infty_P$ to find that (ii) implies (i), see e.g.\ \cite[Definition 1.43 and following discussion]{AB03}.  But as, in the robust setting we envisage, $L^\infty_c$ typically does not possess this property (see for instance Example~\ref{ex:cs2}), we present an alternative argument by means of the following inclusion: 
\begin{equation}\label{trick}i:(L_P^\infty,\sigma(L_P^\infty,L_P^1))\to
(L^1_P,\sigma(L^1_P,L_P^\infty))\end{equation} Note that $i$ is continuous. Now,
as $\Acal$ is Fatou closed, i.e.\ closed under bounded $P$-a.s.\
convergence, it follows that $i(C_K)$ is a closed subset of the
Banach space $(L^1_P, E_P[|\cdot|])$, and thus $i(C_K)$ is also
weakly (i.e.\ $\sigma(L^1_P,L_P^\infty)$) closed by convexity, so
eventually $C_K$ must be weak*-closed by continuity of $i$. 

A natural approach to prove a robust version of
Theorem~\ref{thm:basic} is to 'robustify' the spaces $L^1_P$ and
try to repeat the argument above. There are two natural candidates
for this: Let $H_c:=\{X\in L\mid c(|X|)<\infty \}$, with norm
$\|X\|_c:=c(|X|)$. Then it is readily verified that
$(H_c,\|\cdot\|_c)$ is a Banach lattice. But in the robust case
there is also another candidate, namely
$M_c:=\overline{L^\infty_c}^{\|\cdot\|_c}$ which is also a Banach
lattice with the norm $\|\cdot\|_c$. These spaces have recently
been studied in the literature, see e.g. \cite{DHP11} and
\cite{Nu14}, since they appear as natural environments to embed
financial modelling under uncertainty. Clearly, $L_c^\infty\subset
M_c\subset H_c\subset L_c$. Note that the trick with the inclusion
\eqref{trick} requires that the norm dual of $L^1_P$ can be
identified with $L^\infty_P$, so in particular with a subset of
$L^1_P$ where in this latter case $L^1_P$ is viewed as a
representation of $ca_P$. Thus the reader may readily check that
we could save the above argument if the norm duals $M_c^\ast$ and
$H_c^\ast$ of $M_c$ and $H_c$, respectively, would satisfy
$M_c^\ast\subset ca$ or $H_c^\ast\subset ca$. The following
Theorem~\ref{thm:dual:sigmaadd} shows that this is the case only
if $\Pcal$ is dominated. To this end, denote by
\begin{equation}\label{Z}\Zcal:=\{ (A_n)_{n\in \N}\subset \Fcal\mid
A_n\downarrow \emptyset\; \mbox{and}\; c(A_n)\not \to 0\},
\end{equation} where $A_n\downarrow \emptyset$ means that $A_n\supset A_{n+1}$, $A_n\neq \emptyset$,
$n\in \N$, and $\bigcap_{n\in \N}A_n=\emptyset$, the decreasing sequences of sets on which $c$ is not
continuous.

\begin{theorem}\label{thm:dual:sigmaadd} Consider the following conditions:
\begin{enumerate}
\item $\Zcal=\emptyset$. \item $M_c^\ast\subset ca$. \item
$H_c^\ast\subset ca$.
\end{enumerate}
Then (i) $\Longleftrightarrow$ (ii) $\Longleftarrow$ (iii).

In particular, if $\Zcal=\emptyset$, then there exists a countable
subset $\widetilde \Pcal\subset \Pcal$ such that $\widetilde
\Pcal\thickapprox \Pcal$, and thus there is a probability measure
$Q\in \Mcal_1$ such that $\{Q\}\thickapprox \Pcal $.
\end{theorem}

\begin{proof}
(i) $\Rightarrow$ (ii): By Proposition~\ref{prop:1} for any $l\in
M_c^\ast$ there is $\mu\in ca$ such that $l(X)=\int X\, d\mu$ for
all simple random variables $X$. Moreover, $\mu\in ca_c$, because
$c(A)=0$ implies $l(1_A)=0$, $A\in \Fcal$. Since for any $X\in
L^\infty_c$ and any $n\in \N$ by the usual approximation method
from integration theory there is a simple random variable $X_n$ such that
$|X-X_n|<1/n$ $\Pcal$-q.s., so $\|X-X_n\|_c<1/n$, continuity of
$l$ and the dominated convergence theorem yield $$l(X)=\lim_{n\to
\infty}l(X_n)=\lim_{n\to \infty}\int X_n\, d\mu=\int X\, d\mu$$
for all $X\in L^\infty_c$. We recall that in
\cite{DHP11}~Proposition~18 the following relation was shown
$$M_c= \{X\in H_c\mid \lim_{n\to \infty}\|X1_{\{|X|\geq n\}}\|_c=0 \}.$$
Hence, for $X\in (M_c)_+$ we have by
monotone convergence that $$l(X)=\lim_{n\to \infty}l(X1_{\{|X|\leq
n\}})= \lim_{n\to \infty}\int X1_{\{|X|\leq n\}}\, d\mu = \int X\,
d\mu. $$ Finally, decomposing $X\in M_c$ into $X^+-X^-$ with
$X^+,X^-\in (M_c)_+$ and linearity of $l$ and the integral shows
(ii).

\smallskip\noindent
(ii) $\Rightarrow$ (i) and (iii)  $\Rightarrow$ (i) follow
directly from Proposition~\ref{prop:1}

\smallskip\noindent
The last statement of this theorem is Proposition~\ref{prop:3}.
\end{proof}

\begin{remark}
Note that $\Zcal=\emptyset$ is equivalent to sequential order continuity of $\|\cdot\|_c$. According to Theorem~\ref{thm:dual:sigmaadd},  if $\Pcal$ is not dominated, then $\Zcal\neq \emptyset$ and hence the norm $\|\cdot\|_c$ on $M_c$ or $H_c$ is not order continuous.

Also note that the converse of the last statement of
Theorem~\ref{thm:dual:sigmaadd} is not true, i.e.\ $\Zcal\neq
\emptyset$ does not imply that $\Pcal$ is not dominated. To see
this, let $A_n\downarrow \emptyset$ and pick a sequence of probability measures $P_n$ such that
$P_n(A_n)=1$ for all $n\in \N$, and let $\Pcal=\{P_n\mid n\in
\N\}$. Then, clearly $\|1_{A_n}\|_c=1$ for each $n$. Hence, $\|\cdot\|_c$ is not order continuous and 
$\Zcal\neq \emptyset$ and thus $M_c^\ast\not \subset ca$. However,
we have that $\{Q\}\thickapprox \Pcal$ for
$Q=\sum_{n=1}^\infty\frac{1}{2^n}P_n$.
\end{remark}

Recall the conditions
\begin{description}
\item[(WC)] $\Acal\subset L^\infty_c$ is
$\sigma(L^\infty_c,ca_c)$-closed. \item[(FC)] $\Acal\subset
L^\infty_c$ is Fatou closed: for any bounded sequence $X_n\subset
\Acal$ and $X\in L^\infty_c$ such that $X_n\to X$ $\Pcal$-q.s.\ we
have that  $X\in\Acal$.
\end{description}

It is easily verified that always (WC) $\Longrightarrow$ (FC)
since any bounded $\Pcal$-q.s.\ converging sequence also
converges in $\sigma(L^\infty_c,ca_c)$ to the same limit. However,
there is in general no proof of (FC)
$\Longrightarrow$ (WC) even if $\Acal$ is convex, and also requiring monotonicity of $\Acal$, i.e.\  $\Acal+(L^\infty_c)_+=\Acal$, in addition is not sufficient:

\begin{theorem}\label{thm:4} Let $\Acal\subset L^\infty_c$ be convex and monotone. Without further assumptions on $\Pcal$ or $\Acal$, there exists
no proof of (FC) $\Rightarrow$ (WC).
\end{theorem}

The proof of Theorem~\ref{thm:4} is given by the following
Example~\ref{ex:non} where we give a counter-example of (FC)
$\Longrightarrow$ (WC) assuming the continuum hypothesis. So under
the continuum hypothesis (FC) $\Longrightarrow$ (WC) is indeed
wrong. Note that as the continuum hypothesis does not conflict
with what one perceives as standard mathematical axioms, there is
of course no way to prove (FC) $\Longrightarrow$ (WC) even if we
do not believe in the continuum hypothesis.

\begin{example}\label{ex:non} Consider the measure space $(\Omega,\Fcal)=([0,1],
\mathfrak{P}([0,1])$, where $\mathfrak{P}([0,1])$ denotes the
power set of $[0,1]$. Assume the continuum hypothesis. Banach
and Kuratowski have shown that for any set $I$ with the same cardinality
as $\R$ there is no measure $\mu$ on $(I,\mathfrak{P}(I))$ such
that $\mu(I)=1$ and $\mu(\{\omega\})=0$ for all $\omega\in I$; see
for instance \cite[Theorem~ C.1]{dudley}. It follows that any
probability measure $\mu$ over $(\Omega,\Fcal)$ must be a
countable sum of weighted Dirac-measures, i.e.\
$\mu=\sum_{i=1}^\infty a_i\delta_{\omega_i}$ where $a_i\geq 0$,
$\sum_{i=1}^na_i=1$, $\omega_i\in \Omega$, $i\in \N$. (Recall that
for $\omega\in \Omega$ and $A\in \Fcal$: $\delta_\omega(A)=1$ if
and only if $\omega\in A$ and $\delta_\omega(A)=0$ otherwise.)
Indeed, let $\mu\in \Mcal_1$, and let $$S:=\{\omega\in \Omega\mid
\mu(\{\omega\})>0\}.$$ Then $S$ can at most be countable (consider
the sets $S_n:=\{\omega\in \Omega\mid \mu(\{\omega\})>1/n\}$,
$n\in \N$, and note that $S=\bigcup_{n\in \N}S_n$). Now suppose
that $\mu([0,1]\setminus S)>0$, then as $[0,1]\setminus S$ has the
same cardinality as $[0,1]$, this implies the existence of an atom
for the measure $\mu$ restricted to $[0,1]\setminus S$, i.e.\
there exists $\hat\omega\in [0,1]\setminus S$ such that
$$\frac{1}{\mu([0,1]\setminus S)}\mu(\{\hat\omega\})>0.$$ This
clearly contradicts the definition of $S$.

\smallskip\noindent
Let $\Pcal:=\{\delta_\omega\mid \omega\in [0,1]\}$ be the set of
all Dirac measures.  Then $$c(|X|)=\sup_{\omega\in
[0,1]}|X(\omega)|,$$ so it turns out that
$L^\infty_c=M_c=H_c=\Lcal^\infty$. Hence,
$(L^\infty_c)^\ast=M_c^\ast=H_c^\ast=ba$, and, as $c(A)=0$ is
equivalent to $A=\emptyset$, we also have that $ca_c=ca$. Consider
the set $$C:=\{1_A \mid \emptyset\neq A \subset [0,1]\;\mbox{is countable}\},$$
and let $\Acal$ be the convex closure of $C$ under bounded
$\Pcal$-q.s.\ convergence of sequences. Then $1\not\in \Acal$: Indeed, any
$X=\sum_{i=1}^na_i1_{A_i}$, $a_i\geq 0$, $\sum_{i=1}^na_i=1$,
$1_{A_i}\in C$, in the convex hull of $C$ satisfies $0\leq X\leq
1_{A_X}$ where $A_X:=\bigcup_{i=1}^nA_i$ is countable. Let $X_k$
be any sequence in the convex hull of $C$, then $0\leq  X_k\leq
1_B$, $k\in \N$, where $B:=\bigcup_{k\in \N}A_{X_k}$ is countable.
Hence, $X_k(\omega)=0$ for all $\omega\in [0,1]\setminus B$, so
$1\not\in \Acal$. Now consider the family $\Gcal$ of all countable
subsets of $[0,1]$ directed by $A\leq B$ if and only if $A\subset
B$. Consider the net $\{1_{A}\mid A\in \Gcal\}\subset C$. Then for
any probability measure $\mu$ there is $A\in \Gcal$ (namely $A=S$)
such that for all $B\in \Gcal$ with $B\geq A$ we have $\int1_B\,
d\mu=1=\int 1\,d\mu$.  Thus $1$ lies in the
$\sigma(L^\infty_c,ca_c)$-closure of $\Acal$.

In order to make the presentation simpler, we did not require monotonicity of $\Acal$ so far, but the same arguments as above show that
if $\Acal$ is the convex closure of $-C+(L^\infty_c)_+$ under bounded
$\Pcal$-q.s.\ convergence of sequences, which is convex and monotone, then $-1\not \in  \Acal$ but $-1$ is an element of the $\sigma(L^\infty_c,ca_c)$-closure of $\Acal$.
\end{example}

A consequence of Theorem~\ref{thm:4} is that we need to ask for
additional properties on $\Acal$ in order to have (FC)
$\Longleftrightarrow$ (WC).

\subsection{$\Pcal$-sensitivity, $ca_c^\ast=L^\infty_c$, and (FC) $\Longleftrightarrow$ (WC)}\label{(FC)and(WS)}

A simple property on $\Acal$ which allows to prove (FC)
$\Longleftrightarrow$ (WC) is to require that the convex
set $\Acal\subset L^\infty_c$ behaves as in the dominated case,
i.e.\ there is a reference probability $P\in \Pcal$ such that
$\Acal$ is closed under bounded $P$-a.s.\ convergence. Under this
assumption the whole issue can be reduced to
Theorem~\ref{thm:basic}. Clearly, this assumption is too strong.
However, it gives the idea of the $\Pcal$-sensitivity property we
will introduce in the following.

Given a probability $Q\in \Mcal_1$ such that $\{Q\}\ll \Pcal$ we
define the linear map $j_Q:L^{\infty}_c \to L^{\infty}_Q$ by
$ Q(j_Q(X)=X)=1,
$
i.e.\ $j_Q(X)$ is the equivalence class in $L^\infty_Q$ such that
any representative of $j_Q(X)$ and any representative of $X$ are
$Q$-a.s.\ identical. As $ca_Q$ (which can be identified with
$L^1_Q$) is a subset of $ca_c$, we deduce that $j_Q:
(L^\infty_c,\sigma(L^\infty_c, ca_c))\to (L^ \infty_Q,\sigma(L^
\infty_Q,L^1_Q))$  is continuous.

\begin{definition}\label{sensitive} A set $\Acal\subset L^\infty_c$ is called  \emph{$\Pcal$-sensitive} if there exists a set $\Qcal\subset \Mcal_1$ with  $\Qcal\ll \Pcal$ such
that \begin{equation*}\label{eq:psensitive} \mbox{$j_Q(X)\in
j_Q(\Acal)$ for all $Q\in \Qcal$ implies $X\in \Acal$}
\end{equation*} or equivalently $$\Acal = \bigcap_{Q\in\Qcal}
j^{-1}_Q\circ j_Q(\Acal).$$ The set $\Qcal$ will be called
\emph{reduction set} for $(\Acal,\Pcal)$.
\end{definition}

\begin{remark} Suppose that $\Pcal$ is dominated. Then the Halmos Savage
lemma (see \cite{HS49}, Lemma 7) guarantees the existence  of a
countable subclass $\{P_i\}_{i=1}^{\infty}$ such that
$\{P_i\}_{i=1}^{\infty}\thickapprox \Pcal$. Let $P=\sum
\frac{1}{2^i}P_i$. Then $\Pcal\thickapprox \{P\}$, so the space
$L^{\infty}_c$ can be identified with $L^{\infty}_{ P}$. Hence, in
that case any set $\Acal\subset L^\infty_c$  is automatically
$\Pcal$-sensitive with reduction set $\Qcal=\{P\}$.
\end{remark}

\begin{example}
The set $\Acal$ of Example~\ref{ex:non} is not $\Pcal$-sensitive. Since $c(A)=0$ implies that $A=\emptyset$,
any set of probabilities $\Qcal\subset \Pcal$ satisfies $\Qcal\ll \Pcal$.
Let $Q\in\Mcal_1$ be arbitrary and $S:=\{\omega\in [0,1]\mid Q(\{\omega\})>0\}$ such that  $Q=\sum_{\omega\in S}a_\omega\delta_{\omega}$ with $a_\omega>0$ and $\sum_{\omega\in S}a_\omega=1$. Then
$1_S\in \Acal$ by definition of $\Acal$ and thus $1\in
j_Q(\Acal)$, or to be more precise, $1$ and $1_S$ form the same equivalence class in $L^\infty_Q$. Since $Q\in \Mcal_1$ was arbitrary, we have $1\in
\bigcap_{Q\in \Qcal}j_Q^{-1}\circ j_Q(\Acal)$. As we know that
$1\not\in \Acal$, the set $\Acal$ is not $\Pcal$-sensitive.
\end{example}

Indeed $\Pcal$-sensitivity is a necessary condition for (FC)
$\Longleftrightarrow$ (WC).

\begin{proposition}\label{prop:sensitive}
Any convex set $\Acal\subset L^\infty_c$ which is
$\sigma(L^\infty_c,ca_c)$-closed (i.e.\ satisfies (WC)) is
$\Pcal$-sensitive.
\end{proposition}

\begin{proof} If $\Acal=\emptyset$ or $\Acal= L^{\infty}_c$, the assertion is trivial.
Now assume that $\Acal\neq \emptyset$ and $\Acal\neq
L^{\infty}_c$. As $\Acal$ is $\sigma(L^\infty_c,ca_c)$-closed and
convex, the function
$$\rho(X):=\delta(X\mid \Acal):=
\begin{cases}
0 & \mbox{if} \, X\in \Acal \\ \infty &\mbox{else}
\end{cases}
,\quad X\in L^\infty_c,$$ is convex and
$\sigma(L^\infty_c,ca_c)$ lower-semicontinuous. Hence, by the
Fenchel-Moreau theorem (see \cite[Proposition~4.1]{EKETEM}) there exists
a dual representation of $\rho$, i.e.\
$$\rho(X)=\sup_{\mu\in \Qcal}\left\{\int X\, d\mu-\rho^\ast(\mu)\right\}$$ where
$\Qcal:=\{\mu\in ca_c\mid \rho^\ast(\mu)<\infty\}$ is a convex set
and
$$\rho^\ast(\mu):=\sup_{X\in \Acal}\int X\, d\mu, \quad \mu\in
ca_c .$$
$\Acal\neq L^{\infty}_c$ implies $\Qcal\supsetneqq
\{0\}$ and therefore,
$$\Acal=\bigcap_{\mu\in \Qcal}\left\{X\in L^\infty_c\mid \int X\,
d\mu\leq \rho^\ast(\mu)\right\}=\bigcap_{\mu\in \Qcal\setminus
\{0\}}\left\{X\in L^\infty_c\mid \int X\, d\mu\leq
\rho^\ast(\mu)\right\}.
$$ Let $\tilde \Qcal:=\{\frac{|\mu|}{|\mu|(\Omega)} \mid \mu\in
\Qcal\setminus \{0\}\}\subset \Mcal_1$ and note that $\tilde
\Qcal\ll \Pcal$ since $\Qcal\subset ca_c$. Consider $$X\in
\bigcap_{Q\in \tilde \Qcal}j_Q^{-1}\circ j_Q(\Acal).$$ Fix $Q\in
\tilde \Qcal$ and $\nu\in \Qcal$ such that
$Q=\frac{|\nu|}{|\nu|(\Omega)}$. Then, $j_Q(X)\in j_Q(\Acal)$,
i.e.\ there is $Y\in \Acal$ such that $j_Q(X)=j_Q(Y)$. Noting that
$X=j_Q(X)$ and $Y=j_Q(Y)$ under $\nu$, it follows that $$\int X\,
d\nu=\int j_Q(X)\, d\nu=\int j_Q(Y)\, d\nu=\int Y\, d\nu\leq
\rho^\ast(\nu),$$ where the inequality follows from $Y\in \Acal$.
Since $Q\in \tilde \Qcal$ was arbitrary, we conclude that indeed
$\int X\, d\mu \leq \rho^\ast(\mu)$ for all $\mu\in \Qcal$, and
hence that $X\in \Acal$. This shows that $\bigcap_{Q\in
\Qcal}j_Q^{-1}\circ j_Q(\Acal)\subset \Acal$. The other inclusion
$\bigcap_{Q\in \Qcal}j_Q^{-1}\circ j_Q(\Acal)\supset \Acal$ is
trivially satisfied, so we have that $\Acal$ is $\Pcal$-sensitive
with reduction set $\tilde \Qcal$.
\end{proof}

The following Theorem~\ref{thm:5} gives conditions under which
(FC) $\Longleftrightarrow$ (WC) for a convex set $\Acal\subset
L^\infty_c$. Besides $\Pcal$-sensitivity we have to require that
the norm dual $ca_c^\ast$ of $(ca_c, TV)$, where $TV$ denotes the
total variation norm on $ca_c$, may be identified with
$L^\infty_c$. Clearly any $X\in L^\infty_c$ may be identified with a continuous
linear functional on $ca_c$ by \begin{equation}\label{bal}ca_c\ni
\mu\mapsto \int X \,d\mu,\end{equation} so we always have
$L^\infty_c\subset ca_c^\ast$. However, $ ca_c^\ast=L^\infty_c$ is obviously a
very strong condition which we will characterise in Proposition~\ref{prop:esssup} in terms of order closedness of $L^\infty_c$.

\begin{theorem}\label{thm:5} Suppose that $ca_c^\ast=L^\infty_c$ and let $\Acal\subset L^\infty_c$ be convex and monotone ($\Acal+(L^\infty_c)_+=\Acal$). Equivalent are
\begin{enumerate}
 \item $\Acal$ satisfies (WC).
 \item $\Acal$ is $\Pcal$-sensitive and satisfies (FC).
\end{enumerate}

\end{theorem}

\begin{proof} We already know that (WC) implies (FC) and $\Pcal$-sensitivity.
Now assume that $\Acal$ is $\Pcal$ sensitive and satisfies (FC).
Since $ca_c^\ast=L^\infty_c$, by the Krein-Smulian theorem it is sufficient to show that
$C_K:=\Acal\cap \{Z\in L^\infty\mid\|Z\|_{c,\infty}\leq K\}$ is
$\sigma(L^\infty_c,ca_c)$-closed for every $K>0$. Let
$\Qcal$ be a reduction set for $(\Acal,\Pcal)$ and fix any $K>0$ and $Q\in\Qcal$.

\smallskip\noindent
Consider the continuous inclusion $$i: (L^\infty_Q,\sigma(L^
\infty_Q,L^1_Q))\to (L^1_Q,\sigma(L^1_Q,L^{\infty}_Q)).$$ In a first step we show
that $C_{Q,K}:=i\circ j_Q(C_K)$ is
$\|\cdot\|_Q:=E_Q[|\cdot|]$-closed in $L^1_Q$, because being
convex it then follows that $C_{Q,K}$ is
$\sigma(L^1_Q,L^{\infty}_Q)$-closed and therefore $j_Q(C_K)$ is
$\sigma(L^{\infty}_Q,L^1_Q)$-closed by continuity of $i$. To this
end let $(Y_n)_{n\in \N}\subset C_{Q,K}$ and $Y\in L^1_Q$ such that
$\|Y_n-Y\|_Q\to 0$, and without loss of generality we may also assume that
$Y_n\to Y$ $Q$-a.s. Note that $Y$ is necessarily bounded by $K$. Choose $X_n\in C_K$ such that $Y_n=j_Q(X_n)$
for all $n\in \N$ and $X\in L^\infty_c$ such that $Y=j_Q(X)$.
Consider now the set $$F:=\{\omega\in \Omega\mid
X_n(\omega)\rightarrow X(\omega)\}$$ (by the usual abuse of notation, in the
definition of $F$ we still write $X_n$ and $X$ for arbitrary representatives of
the equivalence classes $X_n$ and $X$).  By monotonicity of
$\Acal$ we have that $\widetilde{X}_n:= X_n 1_F + K1_{F^c}\in C_K$
for all $n\in \N$, and $ \widetilde{X}_n\rightarrow X 1_F
+K1_{F^c}=:\widetilde X$ $\Pcal$-q.s. Consequently $\widetilde X
\in C_K$ and since $Q(F)=1$ we have $Y=j_Q(X)=j_Q(\widetilde X)\in
C_{Q,K}$. Hence, $j_Q(C_{K})$ is $\sigma(L^\infty_Q,L^{1}_Q)$
closed.

\smallskip\noindent
By continuity of $j_Q$, the preimage $j_Q^{-1}\circ j_Q(C_{K})$ is
$\sigma(L^\infty_c,ca_c)$-closed, and as also $\{X\mid
\|X\|_{c,\infty}\leq K\}$ is $\sigma(L^\infty_c,ca_c)$-closed, we
conclude that $$A_{Q,K}:= j_Q^{-1}\circ j_Q(C_{K}) \cap \{X\mid
\|X\|_{c,\infty}\leq K\} \supset C_K$$ and finally also
$\bigcap_{Q\in \Qcal} A_{Q,K}$ are
$\sigma(L^\infty_c,ca_c)$-closed.  Clearly, $\bigcap_{Q\in \Qcal}
A_{K,Q}\supset C_K$. If we can show $\bigcap_{Q\in \Qcal}
A_{Q,K}\subset C_K$, then we are done, because then $\bigcap_{Q\in \Qcal} A_{Q,K}=
C_K$, and thus $C_K$ is $\sigma(L^\infty_c,ca_c)$-closed.
To this end, let $X\in
\bigcap_{Q\in \Qcal} A_{Q,K}$. Then $j_Q(X)\in j_{Q}(\Acal)$ for
any $Q\in \Qcal$ and therefore $X\in \Acal$ by
$\Pcal$-sensitivity. Moreover by definition of $A_{K, Q}$ we also have
$\|X\|_{c,\infty}\leq K$.
\end{proof}

Note that Theorem~\ref{thm:5} proves the so-called $C$-property introduced and discussed in \cite{BF09} for convex and monotone sets.

Let $\Dcal\subset L^\infty_c$. Recall that a supremum of  $\Dcal$
is a least upper bound of $\Dcal$, that is an $X\in L^\infty_c$
such that $Y\leq X$ for all $Y\in \Dcal$, and any
$Z\in L^\infty_c$ such $Y\leq Z$ for all $Y\in
\Dcal$ satisfies $X\leq Z$. The supremum of
$\Dcal$ is denoted by $\operatorname{ess\, sup}_{Y\in \Dcal}Y$. This notation is commonly used in probability theory and it is inspired by the tradition of identifying random variables with the equivalence classes they induce. Indeed for a set of random variables in $\Lcal^\infty$, a supremum in the $\Pcal$-q.s.\ order is only essentially unique---thus called essential supremum ($\operatorname{ess\, sup}$)---in the sense that the equivalence class generated by it in $L^\infty_c$ is unique.  

{\color{red} Note that in the version of this paper published in Positivity the following Proposition~\ref{prop:esssup} contains a flaw. Indeed the proof only shows the following:}

\begin{proposition}\label{prop:esssup}
If $ca_c^\ast=L^\infty_c$ then $L^\infty_c$ is order complete, i.e.\ there exists a supremum for any norm bounded set $\Dcal\subset L^\infty_c$. Conversely, if $L^\infty_c$ is order complete and if the order continuous dual \cite[Definition~1.3.8]{Me91} of $L^\infty_c$ may be identified with $ca_c$, then $ca_c^\ast=L^\infty_c$.
\end{proposition}

\begin{proof}  Recall that $ca$ and thus also
$ca_c$ is an AL-space (\cite[Theorem~10.56]{AB06})

Suppose that $L^\infty_c$ is order complete and that its order continuous dual may be identified with $ca_c$. Then $L^\infty_c$ is in particular also monotonically complete in the sense of \cite[Definition~2.4.18]{Me91}. Thus \cite[Theorem~2.4.22]{Me91} applies which inconjunction with \cite[Theorems~9.22 and 9.33]{AB06} yields $ca_c^\ast=L^\infty_c$.

\smallskip\noindent
In order to prove that $ca_c^\ast=L^\infty_c$ implies the
existence of a supremum for any norm bounded set
$\Dcal\subset L^\infty_c$, note that as $ca_c$ is an AL-space, $ca_c^\ast$ is an AM-space (\cite[Theorem~9.27]{AB06}). In
particular $ca_c^\ast$ is order complete.
Here, the order $\geq_\ast$ on
$ca^\ast_c$ is given by $l\geq_\ast 0$ if and only if $l(\mu)\geq
0$ for all $\mu \in (ca_c)_+$, and a set $\Scal\subset ca_c^\ast$
is order bounded from above if there is $h\in ca_c^\ast$ such that
$ h-l\geq _\ast 0$ for all $l\in \Scal$.  Any norm bounded $\Dcal\subset
L^\infty_c$ is order bounded from above in $ca_c^\ast$, because
$K\mu(\Omega)-\int X\, d\mu\geq 0 $, $\mu\in (ca_c)_+$, for a
constant $K>0$ which is an upper bound of the norm on $\Dcal$, so
$(\mu \mapsto K\mu (\Omega))\in ca_c^\ast$ is an upper bound with
respect to $\geq _\ast$. Thus there is a least upper bound of
$\Dcal$ viewed as a subset of $ca_c^\ast$. Now suppose that
$ca_c^\ast$ can be identified with $L^\infty_c$. Then this least
upper bound of $\Dcal$ may be identified with an element in $X\in
L^\infty_c$, that is $$\int X\, d\mu \geq \int Y\, d\mu\quad
\mbox{for all $\mu \in (ca_c)_+$ and all $Y\in \Dcal$.}$$
Considering measures $\mu$ of type $1_A dP$ for $P\in \Pcal$ and
$A\in \Fcal$ shows that $X\geq Y$ for all $Y\in
\Dcal$, and $\mu \mapsto \int X\, d\mu$ being the least amongst
the upper bounds of $\Dcal$ in the $\geq_\ast$-order implies that
$X$ is a supremum of $\Dcal$.
\end{proof}

\begin{example}\label{ex:cs2}
Recall Example~\ref{ex:non}. 
Clearly any norm bounded set
$\Dcal\subset L^\infty_c=\Lcal^\infty$ admits an essential
supremum which is simply given by $\omega\mapsto \sup_{Y\in
\Dcal}Y(\omega)$. 
%
Assume the continuum hypothesis. Let $l\in ca_c^\ast$ and define
$X(\omega)=l(\delta_\omega)$, $\omega\in [0,1]$. Then by
linearity, for all $\mu\in ca$ it follows that
$l(\mu)=\sum_{\omega\in S}a_\omega l(\delta_\omega)=\int X\, d\mu$
where $S:=\{\omega\in [0,1]\mid \mu(\{\omega\})>0\}$ and
$a_\omega=\mu(\{\omega\})$, $\omega\in S$. 
 Moreover, it is also readily verified that in this case $L^\infty_c$ does not have the countable sup property.
\end{example}

\section{Applications of Theorem \ref{thm:5}}\label{applications}

\subsection{Dual representation of (quasi-) convex increasing
functionals}
In this section we provide a dual representation of
(quasi-) convex increasing functionals.
Such results are key in the study of robustness of financial risk measures. An exhaustive
introduction to the dual representation of convex risk measures
can be found in \cite{FS3} (see also \cite{DK13} for the
quasiconvex case and \cite{CKT15} for recent developments). To the best of our
knowledge, in presence of model uncertainty, the only result
available in the literature is \cite[Theorem 3.1]{Bion} which is
obtained for the closure of the space of continuous functions
under the norm $\|\cdot\|_c$.

\begin{definition} A function
$f:L^\infty_c\to (-\infty,\infty]$ is

\begin{itemize}
\item  quasiconvex (resp. convex) if for every $\lambda\in [0,1]$
and $X,Y\in L^{\infty}$ we have $f(\lambda X+(1-\lambda)Y)\leq
\max\{X,Y\}$ (resp. $f(\lambda X+(1-\lambda)Y)\leq \lambda f(X)+
(1-\lambda)f(Y)$).

\item $\tau$-lower semicontinuous (l.s.c.) for some topology
$\tau$ on $L^{\infty}_c$   if for every $a\in\R$ the lower level
set $\{X\in L^{\infty}_c\mid f(X)\leq a\}$ is $\tau$-closed.

\item $\Pcal$-sensitive if the lower level sets $\{X\in
L^{\infty}_c\mid f(X)\leq a\}$ are $\Pcal$-sensitive for every
$a\in\R$.

\end{itemize}

\end{definition}

The following Lemma provides a huge class of $\Pcal$-sensitive
functions.

\begin{lemma}\label{sup:representation} Consider a function $f:L^\infty_c\to
[-\infty,\infty]$ such that
\begin{equation}\label{sup_property}f(X)=\sup_{P\in\Qcal}
f_P(j_P(X)),
\end{equation}
for some $\Qcal\subset \Mcal_1$ and $f_P:L^{\infty}_P\rightarrow
[-\infty,\infty]$. If $\mathcal{Q}\ll \mathcal{P}$ then $f$ is
$\Pcal$-sensitive with reduction set $\Qcal$.
\end{lemma}

\begin{proof}
From representation \eqref{sup_property} we automatically have
$$\{X\in L^{\infty}_c\mid f(X)\leq a\}=\bigcap_{P\in\Qcal}\{X\in L^{\infty}_c\mid f_P(j_P(X))\leq a\}.$$
As $\{X\in L^{\infty}_c\mid f_P(j_P(X))\leq a\}=
j_P^{-1}\circ j_P\{X\in L^{\infty}_c\mid f_P(j_P(X))\leq a\}$, we conclude that $f$ is $\Pcal$-sensitive with reduction set $\Qcal$.
\end{proof}

\begin{theorem}\label{thm:dual:repr} Assume that $ca_c^\ast=L^\infty_c$. Let $f:L^\infty_c\to
(-\infty,\infty]$ be a  quasiconvex (resp.\ convex), monotone non-decreasing ( $X\leq Y$ $\Pcal$-q.s.\ implies $f(X)\leq f(Y)$) and
$\Pcal$-sensitive function. The following are equivalent:
\begin{itemize}
\item[(i)] $f$ is $\sigma(L^\infty_c,ca_c)$-lower semi continuous.

\item[(ii)] $f$ has the Fatou property: for any bounded sequence
$(X_n)_{n\in\N}\subset L^\infty_c$ converging $\Pcal$-q.s.\ to
$X\in L^\infty_c$ we have $f(X)\leq \liminf_{n\to \infty} f(X_n)$.

\item[(iii)] For any sequence $(X_n)_{n\in \N}\subset \Acal$ and $X\in
L^\infty_c$ such that $X_n\uparrow X$ $\Pcal$-q.s.\ we have that
$f(X_n)\uparrow f(X)$.

\item[(iv)] $f$ admits a bidual representation which in the quasiconvex case is
\begin{eqnarray*}
f(X)&= & \sup_{P\in ca_c\cap
\Mcal_1}R\left(E_P[X],P\right), \quad X\in
L^\infty_c,
\end{eqnarray*}
with  dual function $R:\R\times ca_c\to (-\infty,\infty]$
given by
\begin{eqnarray*}
R(t,\mu)&:= & \sup_{t^{\prime}<t}\inf_{Y\in
L^\infty_c}\left\{f(Y)\mid \int Y\, d\mu= t^{\prime} \right\};
\end{eqnarray*}
and in the convex case the dual representation is
\begin{eqnarray*}
f(X)&= & \sup_{\mu\in
(ca_c)_+}\left\{\int X\, d\mu-f^\ast(\mu)\right\}, \quad X\in
L^\infty_c,
\end{eqnarray*}
where the dual function
 $f^{\ast}:ca_c\to (-\infty,\infty]$) is given by
\begin{eqnarray*}
 f^\ast(\mu)&:= & \sup_{Y\in L^\infty_c}\left\{\int Y\, d\mu-
f(Y)\right\}.
\end{eqnarray*}
\end{itemize}
In addition, if $f(X+c)=f(X)+c$ for every $X\in L^{\infty}_c$ and
$c\in \R$ then $f$ is necessarily convex and
$$f(X)=  \sup_{P\in ca_c\cap
\Mcal_1} \left\{E_P[X]-f^\ast(P)\right\}, \quad X\in
L^\infty_c. $$
\end{theorem}

\begin{proof}
According to Theorem~\ref{thm:5} (i) holds if and only if (ii) is
satisfied.
\\ (ii) $\Rightarrow$ (iii) is due to $$f(X)\leq \liminf_{n\to \infty}f(X_n)\leq f(X)$$ where the last ineqaulity follows from monotonicity.
Conversely  (iii) $\Rightarrow $ (ii) follows by considering
$Y_n:=\operatorname{ess\, inf}_{k\geq n}X_k$ and noting that $Y_n\uparrow X$ $\Pcal$-q.s.\ and
$f(Y_n)\leq f(X_n)$; see also \cite[Lemma~4.16]{FS3}.

In the convex case  $(i) \Leftrightarrow (iv)$ is Fenchel's
Theorem (see \cite[Proposition~4.1]{EKETEM}) together with
monotonicity  (see \cite[Corollary~7]{FR02}).

In the quasiconvex case showing $(i) \Rightarrow (iv)$ is a
consequence of the Penot-Volle duality Theorem (see Appendix~\ref{sec:PenotVolle}) and together with monotonicity (see \cite[Lemma 8]{C3M09}), and $(iv) \Rightarrow (iii)$ follows
from  the monotone convergence theorem and the definition of $R$.
\end{proof}

\subsection{Fundamental Theorem of Asset
Pricing}\label{remarksFTAP}

Pricing theory in mathematical finance is based on the Fundamental
Theorem of Asset Pricing, which roughly asserts that in a market
without arbitrage opportunities (the so-called no-arbitrage
condition) discounted prices are expectations under some
risk-neutral probability measure. This characterisation is
essential to develop a pricing theory for financial instruments
which are not traded in the market.  In the classical dominated
framework on some probability space
 $(\Omega,\Fcal,P)$ the risk-neutral probability measures are martingale measures for the discounted price process which are equivalent to the reference probability $P$, see \cite{DS06} for a detailed review and related literature. Also note that the no-arbitrage condition is necessary and sufficient the existence of an economic equilibrium, see e.g.\ \cite{K81}.

It is well understood that the Fundamental Theorem of Asset
Pricing in a classical dominated framework is highly related to
duality arguments. There are also robust approaches applying duality, see e.g.\  \cite{B13} based on an extended order dual space, the so-called super order dual introduced in \cite{AT01}. 
However, most recent studies of robust Fundamental 
Theorems of Asset Pricing do not use duality arguments given
the difficulties we outlined in this paper, see e.g.\ \cite{BN13}.  
However, under the conditions that we have derived in
Section~\ref{robustFatou} we will see that it is possible to
reconcile the Fundamental Theorem of Asset Pricing, the
Superhedging Duality, and duality theory on the pair $(L^\infty_c,
ca_c)$ using the well-known arguments.

Throughout this section we assume that $ca_c^\ast=L^\infty_c$
holds true. We consider a discrete time market model with terminal
time horizont $T\in \mathbb{N}$, and trading times $I:=\left\{
0,...,T\right\} $. The price process is given by a $\Pcal$-q.s.\
bounded $\mathbb{R}^{d}$-valued stochastic process
$S=(S_{t})_{t\in I}=(S^j_{t})_{t\in I}^{j=1,\ldots,d}$ on $(\Omega
,\mathcal{F})$, and we also assume the existence of a numeraire
asset $S_{t}^{0}=1$ for all $ t\in I$. Moreover, we fix a
filtration $\mathbb{F}:=\{\mathcal{F}_{t}\}_{t\in I}$ such that
the process $S$ is $\mathbb{F}$-adapted. Denote by $\mathcal{H}$
the class of $\R^d$-valued, $\mathbb{F}$-predictable stochastic
processes, which is the class of all admissible trading
strategies. Let
$$\Ccal:=\left\{ X\in L^\infty_c\mid X\leq (H\bullet S)_T \;\Pcal\text{-q.s. for some } H\in
\Hcal\right\}$$ where $$(H\bullet
S)_t:=\sum_{k=1}^{t}\sum_{j=1}^{d}H_{k}^{j}(S_{k}^{j}-S_{k-1}^{j})
$$ is the payoff of the self-financing trading strategy at time $t\in I\setminus \{0\}$ with
initial investment $(H\bullet
S)_0=0$ given by the predictable process
$H=(H_t)_{t\in I\setminus\{0\}}$. In this framework the no-arbitrage condition
(NA($\Pcal$)) was introduced by \cite{BN13} as given by the
following definition.
\begin{definition}
The described market model is called arbitrage-free, if it
satisfies the no-arbitrage condition
\begin{description}
\item[NA($\Pcal$)] $(H\bullet S)_T\geq 0$ $\Pcal$-q.s.\ implies
$(H\bullet S)_T= 0$ $\Pcal$-q.s..
\end{description}
\end{definition}

Note that NA($\Pcal$) is equivalent to $\Ccal\cap
(L^\infty_c)_+=\{0\}$.

\begin{lemma}\label{C-closed} Under $NA(\Pcal)$ if $\mathcal{C}$ is
$\Pcal$-sensitive then  $\mathcal{C}$ is
$\sigma(L^{\infty}_c,ca_c)$-closed.
\end{lemma}

\begin{proof} \cite[Theorem 2.2 ]{BN13}   shows
that under $NA(\Pcal)$ the cone $\Ccal$ is closed under $\Pcal$-q.s. convergence of
sequences and therefore $\mathcal{C}$ satisfies (FC). We remark that \cite[Theorem 2.2]{BN13} holds in full generality
without the product structure on the underlying probability space assumed in \cite{BN13}. Therefore applying Theorem
\ref{thm:5} we deduce that $\mathcal{C}$ is
$\sigma(L^{\infty}_c,ca_c)$-closed.
\end{proof}

Suppose that $\Ccal$ is $\Pcal$-sensitive. As $\Ccal$ is a $\sigma(L^{\infty}_c,ca_c)$-closed convex cone, the bipolar Theorem yields
\begin{eqnarray}\mathcal{C} & = & \mathcal{C}^{00}=\left\{Y\in L^{\infty}_c\mid \forall Q\in \mathcal{C}^0_1: \,E_Q[ Y ] \leq 0\right\} \label{eq:bipolar}
\\ \text{ where } \mathcal{C}^0_1 & := & \mathcal{C}^{0}\cap \Mcal_1=\left\{\mu\in \mathcal{C}^{0}\mid \mu(1_{\Omega})=1\right\}\nonumber
\\ \text{ and } \mathcal{C}^0 & := & \left\{\mu\in ca_c\mid  \forall X\in \mathcal{C}: \, \int X\,  d\mu \leq 0\right\}. \nonumber
\end{eqnarray}
Notice that since $\mathcal{C}\supset -(L^{\infty}_c)_+$ then
$\mu\in (ca_c)_+$ for every $\mu\in \mathcal{C}^0$ which explains $\mathcal{C}^0_1$.

\begin{lemma}\label{lem:arb} $\mathcal{C}^0_1$ is the set of all martingale measures
dominated by the capacity $c$, that is
$$\mathcal{C}^0_1 = \{ Q\ll \Pcal \mid S \text{ is a } Q\text{-martingale}  \} $$
\end{lemma}

\begin{proof}The proof is well-known and straightforward, so we just
give the basic arguments: indeed choose any $Q\in \{ Q\ll \Pcal
\mid S \text{ is a } Q\text{-martingale}  \}$, and let $X\in
\mathcal{C}$ and $H\in \mathcal{H}$ such that $X\leq (H\bullet
S)_T$ $\Pcal$-q.s. Then $E_Q[X]\leq E_Q[(H\bullet S)_T]= (H\bullet
S)_0=0$ since $((H\bullet S)_t)_{t\in I}$ is a $Q$-martingale
(using generalized conditional expectations, see
\cite[Appendix]{BN13}). Thus $Q\in \mathcal{C}_1^0$.

\smallskip\noindent
If $Q\in\mathcal{C}_1^0$ then $E_Q[(H\bullet S)_T]=0$ for any
$H\in \mathcal{H}$ and by choosing appropriate strategies in
$\mathcal{H}$ such as
$H_t^{j}=1_A$ for  $A\in
\Fcal_{t-1}$, $H_t^i=0$ for $i\neq j$ and $H_s=0$ for $s\neq t$
one verifies that $Q$ is a martingale measure for $S$.
\end{proof}

\begin{theorem}[First Fundamental Theorem of Asset Pricing]\label{1FTAP}$\,$\\ Suppose $\Ccal$ is $\Pcal$-sensitive. The following are equivalent:

\begin{enumerate}
\item $NA(\Pcal)$
\item $\mathcal{C}^0_1\approx \Pcal$
\end{enumerate}
Moreover, the Superhedging Duality holds, that is for any $X\in L^{\infty}_c$ the minimal superhedging price \begin{equation*}
\pi(X):=\inf \left\{ x\in \mathbb{R}\mid \exists H\in \mathcal{H}\text{ s.t. }x+(H\bullet S)_{T} \geq X \,  \Pcal\text{-q.s.}\right\}
\end{equation*}
satisfies \begin{equation}\label{dual:super}
\pi(X)=\sup_{Q\in \mathcal{C}^0_1}E_{Q}[X].
\end{equation}
\end{theorem}

\begin{proof}(i) $\Rightarrow$ (ii): Clearly, $c(A)=0$ implies $\sup_{Q\in \mathcal{C}^0_1}Q(A)=0$ as $\mathcal{C}^0_1\subset
ca_c$. Let $B\in \Fcal$  such that
$Q(B)=0$ for all $Q\in\mathcal{C}^0_1$. Thus $1_B\in \Ccal$ by \eqref{eq:bipolar}, so $1_B=0$ in $L^\infty_c$ by $NA(\Pcal)$, i.e.\ $c(B)=0$.

\smallskip\noindent
(ii) $\Rightarrow$ (i): let $H\in\Hcal$ such that $(H\bullet S)_T\geq 0$
$\Pcal$-q.s. Then $Q\{(H\bullet S)_T\geq 0\}=0$ for every $Q\in
\mathcal{C}^0_1$, because $(H\bullet S)_t$ is a $Q$-martingale with expectation $0$, and therefore $(H\bullet S)_T= 0$ $\Pcal$-q.s.

\smallskip\noindent
As for the Superhedging Duality note that clearly $\pi(X)\leq \|X\|_{c,\infty}$ since $0\in \Hcal$, and as $\Ccal_1^0\neq \emptyset$ ($\Ccal\neq L^\infty_c$) it follows that $\pi(X)>-\infty$. Moreover, by \eqref{eq:bipolar} we have for any $y\in \R$ that $X-y\in \Ccal$ if and only if $0\geq \sup_{Q\in \mathcal{C}^0_1}E_{Q}[X-y]= -y + \sup_{Q\in \mathcal{C}^0_1}E_{Q}[X]$ which proves \eqref{dual:super}.
\end{proof}

\begin{appendix}

\section{Auxiliary results for Theorem~\ref{thm:dual:sigmaadd}}

Recall the set $\Zcal$ defined in \eqref{Z}.

\begin{proposition}\label{prop:3}
If $\Zcal=\emptyset$, then there exists a countable subset
$\widetilde \Pcal\subset \Pcal$ such that $\widetilde
\Pcal\thickapprox \Pcal$. The latter implies that there is a probability
measure $Q\in \Mcal_1$ such that $\{Q\}\thickapprox \Pcal $.
\end{proposition}

\begin{proof}
We claim that for each $\varepsilon>0$, there exists $P_1,\ldots,
P_n\in \Pcal$ and $\delta>0$ such that $P_i(A)<\delta$ for all
$i=1,\ldots, n$ implies that for all $P\in \Pcal$ we have
$P(A)<\varepsilon$. Suppose this is not the case. Then there
exists $\varepsilon>0$ such that for any $P_1\in \Pcal$ there is
$A_1\in \Fcal$ and $P_2\in \Pcal$ satisfying $$P_1(A_1)<1/2 \quad
\mbox{and}\quad  P_2(A_1)\geq \varepsilon.$$ Then there also
exists $A_2\in \Fcal$ and $P_3\in \Pcal$ such that
$$P_1(A_2)<1/4,\, P_2(A_2)<1/4   \quad \mbox{while}\quad
P_3(A_2)\geq \varepsilon.$$ Continuing this procedure we find
sequences $(A_n)_{n\in \N}\subset \Fcal$ and $(P_{n})_{n\in \N}\in
\Pcal$ such that
$$P_i(A_n)<\frac{1}{2^n}, \, i=1,\ldots, n, \quad \mbox{and}\quad
P_{n+1}(A_n)\geq \varepsilon.$$ Consider $N:=\bigcap_{n\in
\N}\bigcup_{k\geq n} A_k$. Then $P_i(N)=0$ for each $i\in \N$,
because for all $n>(i-1)$ $$P_i(N)\leq \sum_{k=n}^\infty
P_i(A_k)\leq \frac{1}{2^{n-1}}.$$ Hence, replacing the above
sequence $A_n$ by $B_n:=A_n\setminus N$,
 $n\in \N$, we still have  $$P_i(B_n)<\frac{1}{2^n}, \, i=1,\ldots, n, \quad \mbox{and}\quad P_{n+1}(B_n)\geq \varepsilon.$$
Now let $E_n:=\bigcup_{k\geq n}B_k$, $n\in \N$. It follows that
$E_n\downarrow \emptyset$. However, for each $n\in \N$
$$c(E_n)\geq P_{n+1}(E_n)\geq P_{n+1}(B_n)\geq \varepsilon$$ which
contradicts $\Zcal=\emptyset$.

\smallskip\noindent
Now let $\delta_n>0$ and let $P^{(n)}_1,\ldots, P^{(n)}_{m(n)}\in
\Pcal$ be such that  for all $P \in
\Pcal$ it holds $P(A)<1/n$ whenever $P^{(n)}_i(A)<\delta_n$
for all $i=1,\ldots, m(n)$. Define $$\mu:=\sum_{n=1}^\infty
\sum_{i=1}^{m(n)}\frac{1}{2^n}\frac{1}{2^i}P_i^{(n)}.$$ Then
$\mu\in ca_+$, and $\mu(A)=0$ implies that $P_i^{(n)}(A)=0$ for
all $i=1,\ldots, m(n)$ and $n\in \N$. Eventually this implies that
for all $P\in \Pcal$ we have $P(A)<1/n$ for all $n\in \N$, hence
$P(A)=0$. Thus $$\widetilde \Pcal:=\{P_i^{(n)}\mid i\in
\{1,\ldots, m(n)\}, n\in \N\}\quad \mbox{and}\quad
Q:=\frac{1}{\mu(\Omega)}\mu$$ satisfy the assertion.
\end{proof}

\begin{proposition}\label{prop:1}
Let $(B,\|\cdot\|)$ be a Banach lattice of (equivalence classes
of) random variables on $(\Omega,\Fcal)$  containing all simple
random variables such that the order $\leq$ on $B$ satisfies
$0\leq 1_A\leq 1_{A'}$ whenever $A\subset A'$ for $A,A'\in\Fcal$.
If $B^\ast\subset ca$, in the sense that every $l\in B^\ast$ is of
type $$l(X)=\int X\,d\mu, \quad X\in B,$$ for some $\mu\in ca$,
then $\|1_{A_n}\|\to 0$ $(n\to \infty)$ for all $(A_n)_{n\in
\N}\subset \Fcal$ such that $A_n\downarrow \emptyset$.\\
Conversely, if $\|1_{A_n}\|\to 0$ $(n\to \infty)$ for all
$(A_n)_{n\in \N}\subset \Fcal$ such that $A_n\downarrow
\emptyset$, then for every $l\in B^\ast$ there is a $\mu\in ca$
such that $l(Y)=\int Y\, d\mu $ for all simple random variables
$Y$.
\end{proposition}

\begin{proof}
Suppose that $B^\ast\subset ca$ and let $(A_n)_{n\in \N}\subset
\Fcal$ such that $A_n\downarrow \emptyset$. Then $1_{A_n}\to 0$
with respect to $\sigma(B,B^\ast)$ since every element in $B^\ast$
corresponds to a $\sigma$-additive measure. Hence, $$0\in
\overline{co\{1_{A_n}\mid n\in \N\}}$$ where the closure is taken
in the $\sigma(B,B^\ast)$-topology. As the closed convex set in
the $\sigma(B,B^\ast)$-topology and in the norm topology coincide,
we have that there is a sequence of convex combinations
$$c_k:=\sum_{i=1}^{m(k)}a_i(k)1_{A_{n_i(k)}}, \quad k\in \N,$$
where $a_i(k)\in \R$ and $n_1(k)\leq n_2(k)\leq \ldots \leq
n_{m(k)}(k)$ for all $k\in \N$ such that $\|c_k\|\to 0$ for $k\to
\infty$. Moreover, since $0\in\overline{co\{1_{A_n}\mid n\geq
N\}}$ for any $N\in \N$, we may assume that $n_1(k)\leq n_1(k+1)$
for all $k\in \N$. However, $c_k\geq 1_{A_k}$ where
$A_k=A_{n_{m(k)}(k)}$, because $A_n\supset A_{n+1}$ for all
$n\in\N$. Thus, as $\|\cdot\|$ is a lattice norm, the subsequence
$1_{A_ k}$ converges to $0$ in norm and hence also $1_{A_{n}}$
converges to $0$ in the norm topology (again due to $A_n\supset
A_{n+1}$ for all $n\in\N$).

\smallskip\noindent
Finally suppose that $\|1_{A_n}\|\to 0$ $(n\to \infty)$ for all
$(A_n)_{n\in \N}\subset \Fcal$ such that $A_n\downarrow
\emptyset$. Then for any $l\in B^\ast$, the set function
$$\mu(A):=l(1_A), \quad A\in \Fcal,$$ is $\sigma$-additive. By
linearity of $l$ we deduce that $l(X)=\int X\, d\mu$ for all
simple random variables $X$.
\end{proof}

\section{Penot--Volle Duality Theorem}\label{sec:PenotVolle}
\begin{theorem}(see e.g.\ \cite[Theorem 1.1]{FM11})
\label{Volle1}Let $L$ be a locally convex topological vector space, $%
L^{\prime }$ be its dual space and $f:L\rightarrow \overline{\mathbb{R}}:=%
\mathbb{R}\cup \left\{ -\infty \right\} \cup \left\{ \infty
\right\} $ be quasiconvex and lower semicontinuous. Then
\begin{equation}
f(X)=\sup_{X^{\prime }\in L^{\prime }}R(X^{\prime }(X),X^{\prime
}) \label{111}
\end{equation}%
where $R:\mathbb{R\times }L^{\prime }\rightarrow
\overline{\mathbb{R}}$ is defined by
\begin{equation}
R(t,X^{\prime }):=\inf_{\xi \in L}\left\{ f(\xi )\mid X^{\prime
}(\xi )\geq t\right\} . \label{112}
\end{equation}
\end{theorem}
\end{appendix}

\end{document}